\documentclass[11pt]{amsart}

\usepackage{amsmath,amsthm,amssymb,mathrsfs,amsfonts,verbatim,color}
\usepackage[all,tips]{xy}
\usepackage{graphicx,ifpdf}
\ifpdf
\fi

\newtheorem{theorem}{Theorem}[section]
\newtheorem{lemma}[theorem]{Lemma}

\newtheorem{proposition}[theorem]{Proposition}

\theoremstyle{definition}

\newcommand{\SL}{\operatorname{SL}}
\newcommand{\GL}{\operatorname{GL}}
\newcommand\eq[2]{{\begin{equation}\label{eq:
#1}{#2}\end{equation}}}
\newcommand {\equ}[1] {\eqref{eq: #1}}

\newcommand{\df}{{\, \stackrel{\mathrm{def}}{=}\, }}
\newcommand\ssm{\smallsetminus}
\newcommand{\Q}{{\mathbb {Q}}}

\newcommand {\new}[1] {\textcolor{black}{#1}}

\theoremstyle{remark}
\newtheorem{remark}[theorem]{Remark}

\numberwithin{equation}{section}

\begin{document}

\title{$S$-adic version of  Minkowski's Geometry of numbers and Mahler's compactness criterion}

\author{Dmitry Kleinbock}
\address{Department of Mathematics, Brandeis University, Waltham, MA 02454-9110, USA}
\email{kleinboc@brandeis.edu}

\author{Ronggang Shi}
\address{School of Mathematical Sciences, Xiamen University, Xiamen 361005, PR China}
 \email{ronggang@xmu.edu.cn}
\thanks{The first-named author was supported by NSF grant  DMS-1600814. The second-named author is supported  by  Fundamental Research Funds for the Central Universities (Grant No. 20720160006). The third named author acknowledges partial support by IMI of BAS}

\author{George Tomanov}
\address{Institut Girard Desargues, Universit\'e Claude Bernard-Lyon 1, 69622 Villeurbanne cedex, France}
\email{Georges.Tomanov@desargues.univ-lyon1.fr}


\date{}



\begin{abstract}
\noindent
In this note  we give a detailed proof of certain results on geometry of numbers in
the $S$-adic case. These results are well-known to experts, so the aim here is to provide a
convenient reference for the people who need to use them.
\end{abstract}

\maketitle

\markright{}

\section{Introduction}
The space  of unimodular   lattices in $\mathbb R^n$ $(n\ge 2)$  can be identified with the homogeneous
space $X=\SL_n(\mathbb Z)
\backslash \SL_n(\mathbb R)$ via the correspondence  $\mathbb Z^ng\leftrightarrow \SL_n(\mathbb Z)g$ where $g\in \SL_n(\mathbb R)$. It is proved by Mahler \cite{M2} that a subset $R$ of
$X$ is relatively compact if and only if nonzero elements of
  the corresponding unimodular lattices are separated from zero.
This phenomenon is called Mahler's compactness criterion \new{\cite[Chapter V]{CasselsGN}. It has been very  useful in  dynamical
approach to number theory; we} refer the readers to survey papers \cite{E},\cite{EL} and \cite{KSS} and references
there for details.

Let  $S$ be a finite nonempty set of
places of a global field $K$. We assume $S$ contains all the archimedean places if $K$
is a number field.  For each place $v$ of $K$, let $K_v$ be the completion of $K$ at $v$.
Let  $K_S=\prod_{v\in S}K_v$  and
\begin{equation}\label{eq: integer}
I _S=\{a\in K: a \mbox{ is integral in } K_v \mbox{ for every place } v\not \in S\}.
\end{equation}
We consider $K$ and hence $I_S$ as subrings of $K_S$ via natural embeddings $K\to K_v$.
Then the homogeneous space $\SL_n(I_S)\backslash \SL_n(K_S)$ can be identified with a  set  of
free discrete $I_S$-modules  of rank $n $ in $K_S^n$ with fixed covolume.
The connection between dynamics and number theory also  spreads to the  $S$-adic setting,
where the corresponding version of Mahler's criterion plays an important role.
The extension of Mahler's criterion to the $S$-adic case when $K$ is a number field has already been
 used
 in several papers, and a proof for $K=\mathbb Q$  can be found in \cite{KT}. \new{Moreover, a preliminary version \cite{KTMP}
 	of the paper  \cite{KT}, published as a preprint of MPIM (Bonn), contains a proof of the $S$-adic Mahler's criterion for arbitrary number field $K$.}
	When $K$ is a function field with genus zero and $S$ contains a single place of degree one, Mahler's criterion is proved  in \cite{G}.
The general $S$-adic case is  known to experts, but  it is not easy to   find
 a convenient reference.
Here we provide a self-contained proof of an $S$-adic version of  Mahler's criterion.
\begin{theorem}\label{thm: intro main}
	Let $n\ge 2$.
A set $R\subset \SL_n(I_S)
\backslash \SL_n(K_S)$ is relatively compact if and only if the  subset
\[
\big\{\xi\in I_S^ng: \xi \neq  0,\  g\in \SL_n(K_S) \mbox { and } \SL_n(I_S)g \in R\big\}
\]
of  $K_S^n$
is separated from zero, i.e.~this set has empty intersection with
some open neighborhood of zero in $K_S^n$.
\end{theorem}

Our proof of  Theorem \ref{thm: intro main} is based on
 an $S$-adic version of  Minkowski's lemma of geometry  of numbers.
 Let $\mathbf{vol}$ be the normalized Haar measure on
the additive group $K_S^n$ (see \S \ref{sec: preliminaries}).
For a discrete $I_S$-module $\Gamma$ in $K_S^n$ the covolume of $\Gamma$
(denoted by $\mathbf{cov}(\Gamma)$) is the $\mathbf{vol}$ of a
fundamental domain of $\Gamma$ in $K_S^n$.
Let $B_r(K_S^n)$ be the closed  ball of radius $r$ centered at zero in $K_S^n$ with respect to the normalized  norm
(see \S \ref{sec: preliminaries}).
For each integer $1\le m\le n$, the $m$-th  minimum of a discrete $I_S$-module $\Gamma$ is defined by
\begin{equation}\label{eq: intro minimum}
\lambda_m(\Gamma)=\inf \{r>0: \mathbf{dim}_K\big(\mathbf{span}_K(B_r({K_S^n})\cap \Gamma)\big )\ge m\}.
\end{equation}
Here $\mathbf{span}_K$ is the $K$ linear span of a set  and $\mathbf{dim}_K$ is the dimension of a vector space over $K$.  Similar notations are used when $K$ is
replaced by other rings.
We remark here  that if $K=\mathbb Q$ and $S$ contains only the archimedean place, then we  get the usual concept of successive minima of
lattice points in $\mathbb R^n$.

For two nonnegative real numbers $s$ and $t$ the notation $s \asymp t$ means $C^{-1}s\le t\le Cs$ for some
constant $C\ge 1$.
Let $\sigma$ and $\tau$ be the number of real and complex
places of $K$ respectively.\footnote{If $K$ is function field we have $\sigma=\tau=0$.}
Let $\sharp S=\tau +\mathbf{card}(S)$
where  $\mathbf{card} $ denotes the cardinality of a set.
The $S$-adic version of  Minkowski's theorem on successive minima  (see  \cite[ Chapter IV, \S 1]{S} for the usual case)
is the following theorem.

\begin{theorem}\label{thm: Minkowski}
	
Let $n\ge 1$  and let   $\Gamma\subset K_S^n$ be a discrete $I_S$-module with finite covolume. Then
\[
 \big(\lambda_1(\Gamma)\dots \lambda_n(\Gamma)\big)^{\sharp S}\asymp \mathbf{cov}(\Gamma)
\]
where the implied constants  depend on $K,S$ and $n$.
\end{theorem}

A refine\new{d} version of Theorem \ref{thm: Minkowski} will be proved
in Theorem \ref{thm: refine} where the implied constants \new{will be} explicitly calculated.
If $K$ is a function field of genus zero and $S$ consists of  a single place  of degree one,
then Theorem \ref{thm: Minkowski} is established in \cite{M}. The adelic versions of Theorem
\ref{thm: Minkowski} \new{are} proved in \cite{BV}  (resp.~\cite{T})
  when $K$ is a   number field  (resp.~function field).



\section{Preliminaries: notations}\label{sec: preliminaries}

Let $ K$ be a global field and let $P$ be the set of places of $K$.
Throughout this paper we fix a
 positive integer $n$ and a
finite nonempty set
$S\subset P$  such that $S\supset P_0$ where   $P_0$ (possibly empty) is the set of  archimedean places of $K$.

For every $v\in P$ let $K_v$ be the completion of $K$ at $v$.
 The {\sl $S$-adic numbers} and {\sl integers} are defined  as
 \[
 K_S\new{\df}\prod_{v\in S} K_v\quad\mbox{and}\quad   I _S\new{\df}\{x\in K: x \mbox{ is integral for all } v\in P\ssm S\}
 \]
 respectively.
  We consider $ K$ as a subring of $K_S$ via the natural inclusions  $K\to K_v$.
 For $v\in P$, let $|\cdot|_v$ be the {\sl normalized absolute value}  on $K_v$.
 If $v$ is archimedean, 
 we identify $K_v$ with real or complex numbers where the usual absolute value is $|\cdot|_v$. If $v$ is ultrametric
 then $|a|_v^{-1}= \mathbf{card} (I_v/aI_v)$ for all $a\in I_v$ where
 $I_v$ is the ring of integers of $K_v$.
 For each ultrametric place $v\in P$ we fix a {\sl uniformizer} $\varpi_v$ (a generator of the maximal ideal
 of $I_v$) and take $q_v=|\varpi_v|^{-1}_v$.
We define the {\sl absolute value} and {\sl content} for $x=(x_v)_{v\in S}\in K_S$ respectively  by
 \[
 |x|\new{\df}\max_{v\in S} |x_v|_v \quad \mbox{and}\quad\mathbf {cont} (x)\new{\df}\prod_{v\in S}|x_v|_v^{\varepsilon_v}
 \]
 where $\varepsilon_v=2$ if $K_v=\mathbb C$ and $\varepsilon_v=1$ otherwise.

The additive group  $K_S^n $ can be naturally identified with $\prod_{v\in S} K_v^n$ and
we write every $\xi \in K_S^n$ as $(\xi_v)_{v\in S}$ according to this identification.
More precisely,  if $\xi= (x_1,\ldots, x_n)$ where $x_i=(x_{i,v})_{v\in S}$, then
$\xi_v=(x_{1,v}, \ldots, x_{n,v})$. Similarly,
the   group  $\GL_n(K_S)$ can be naturally identified with $\prod_{v\in S}\GL_n(K_v) $ and we write
 every  $g\in \GL_n(K_S)$
 as $(g_v)_{v\in S}$ according to this identification.
The group $\GL_n(K_v)$ (resp.~$\GL_n(K_S)$)  acts
on  $K_v^n$ (resp.~$K_S^n$) by matrix multiplication from the right.
Moreover, the action of $g\in \GL_n(K_S)$ on $\xi\in K_S^n$ is consistent with these identifications, that is,
$\xi g=  ( \xi_v g_v)_{v\in S}$ under previous notations.

For  $v\in P $ we take $\mathbf{vol}_{v}$ to be
the normalized Haar measure on
$K_v$. For archimedean $v$, the measure  $\mathbf{vol}_{v}$ is the Lebesgue measure.
 If $v$ is ultrametric, the measure
satisfies  $\mathbf {vol}_v (I_v^n)=1$.
It follows directly from definition that
\[
\mathbf{vol}_{v}(aB)=|a|_{v}^{n\varepsilon_v}\mathbf {vol }_{v}(B)
\]
for every $a\in K_v$ and {\new{any} measurable subset $B$ of $K_v^n$.
We take the  normalized Haar measure  $\mathbf{vol}$  on $K_S^n $
to be the
product measure.     In the sequel we  will abbreviate  $\mathrm{d}\mathbf {vol}(\xi)$  by $\mathrm{d}\xi$ for the integration
   with respect to the volume measure.
   For a positive integer $m $, we use
   $\mathbf{vol}_v^m$ and $\mathbf{vol}^m$  to denote the normalized Haar measures on
   $K_v^m$ and $K_S^m$ respectively.

 If $v$ is  archimedean  we take
 $\|\cdot\|_v$  to be the Euclidean norm on   $K_v^n$.
If $v$ is ultrametric we take
$\|\cdot\|_v$ to be the sup norm with respect to coordinates, that is
\[
\|(a_1,\dots, a_n)\|_v\new{\df}\max_{1\le i\le n}|a_i|_v\quad \mbox{where}\quad a_i\in K_v.
\]
We define the {\sl norm} and {\sl content}  for  $\xi\in K_S^n$ by
\[
\|\xi\|\new{\df}\max_{v\in S}\|\xi_v\|_v\quad
\mbox{and}\quad \mathbf{cont}(\xi)\new{\df}\prod_{v\in S}\|\xi_v\|_v^{\varepsilon_v}.
\]
For  $a\in I_S^*$, where $I_S^*$ is the group of multiplicatively invertible elements  of $I_S$,  we
have that
 $\mathbf{cont}(a)=1$ and
$\mathbf{cont}(a\xi)=\mathbf{cont}(\xi)$ where $\xi\in K_S^n$. Also for every $g\in \GL_n(K_S)$ we have
 \begin{align}\label{eq;det}
 \mathrm{d}(\xi g)=\mathbf{cont}\big(\mathbf{det}(g)\big)\,\mathrm{d}\xi,
 \end{align}
  where $\mathbf{det}$ is the determinant of a matrix.

The set of  vectors in $K_S^n$ (resp.~$K_v^n$) with norm less than or equal to $r$ is denoted by $B_r(K_S^n)$  (resp.~$B_r(K_v^n)$).
It can be checked directly  that
\begin{align}\label{eq;ball}
B_r(K_S^n)=\prod_{v\in S} B_r(K_v^n).
\end{align}
Let $L$ be a free $K_S$-submodule of $K_S^n$ with rank $m\le n$. Then
\begin{align}\label{eq;decomposition}
L=\prod_{v\in S} L_v,
\end{align} where $L_v$
is an $m$-dimensional subspace of $K_v^n$.
There is a unique additive Haar measure
$\mathbf{vol}_L$
on $L$ (resp.~$\mathbf{vol}_{L_v}$ on $L_v$) such that
\begin{align*}
\mathbf {vol}_L\big(L\cap B_1(K_S^n)\big)&=\mathbf {vol}^{m}\big(B_1(K_S^m)\big) \\
 (\mbox{resp.}\quad \mathbf {vol}_{L_v}\big(L\cap B_1(K_v^n)\big)&=\mathbf {vol}^m_{v}\big(B_1(K_v^m)\big)\ ).
\end{align*}
Moreover, the above definition, (\ref{eq;ball}) and (\ref{eq;decomposition})  imply
 \begin{align}\label{eq;measure}
 \mathbf {vol}_L=\prod_{v\in S}\mathbf {vol}_{L_v}.
\end{align}
 In the case where
$K=\mathbb Q$ and $\mathbf{card} (S)=1$, the measure $\mathbf{vol}_L$ is the measure given by the inner product on $L$.
Suppose $\xi=(\xi_v)_{v\in S}$ and $\xi_v \neq 0$; \new{then}  the covolume of $I_S \xi$ ($I_S$-linear span of $\xi$) in $K_S \xi$ ($K_S$-linear span of $\xi$)
with respect to $\mathbf{vol}_{K_S \xi}$ is equal to
$\mathbf{cont}(\xi) $ multiplied by the covolume of $I_S$ in $K_S$.
The covolume of
 a   discrete $I_S$-module  $\Gamma$ in  $K_S^n$
 with respect to
the induced measure $\mathbf{vol}_{K_S\Gamma}$ is called {\sl relative covolume} of $\Gamma$  and
it is denoted by
$\mathbf{cov_r}(\Gamma)$.  The covolume of  $\Gamma$ with respect to $\mathbf{vol}$ is denoted by
$\mathbf{cov}(\Gamma)$.


\section{Discrete $I_S$-modules}\label{sec: discrete}
Let $\Gamma\subset K_S^n$ be a discrete $I_S$-module.
In this section, we use ideas of
 \cite[\S 8]{KT}  to study properties of $\Gamma$.

\begin{lemma}\label{lem: discrete module}
Let $\Gamma\subset K_S^n$ be a discrete  $ I _S$-module
 and let $ \xi_1,\dots,  \xi_m\in \Gamma$. The following statements are equivalent:
\begin{enumerate}
\item  $\xi_1,\dots, \xi_m$ are linearly independent over $ I_S$;
\item $\xi_1,\dots, \xi_m$ are linearly independent over $ K$;
\item $\xi_1,\dots, \xi_m$ are linearly independent over $ K_S$.
\end{enumerate}
\end{lemma}

\begin{proof}
It suffices to  show that (1) implies (3).  We prove it by
 induction on
$m$.
Write $\xi_i=(\xi_{i,v})_{v\in S}$ as in \S \ref{sec: preliminaries}. Suppose that
 $\xi_1$ is linearly dependent over $K_S$, then there exists $w\in S$ such that $\xi_{1,w}=0$.
According to the strong approximation theorem (see \cite[Chapter II \S 15]{CF}), there is a sequence
 $\{ c_i \}_{i\ge 1}$ of $  I_S\smallsetminus \{0 \}$
such that $|c_{i}|_v\to 0$ as $i\to \infty$ for any $v\in S\smallsetminus \{w \} $. Therefore $c_i\xi_1\to 0$ which contradicts
the assumption that $\Gamma $ is discrete. This proves (3) in the  case where  $m=1$.

Now suppose $m>1$ and   (1) implies (3) while $m$ is replaced by $m-1$. By the case for $m=1$, we know $\xi_{1,v}\neq 0$ for every
$v\in S$. So there exists $g\in \GL_n(K_S)$ such that $\xi_1g=(1, 0, \dots, 0)$. The right multiplication of
 $g$ on $K_S^n$ is a $K_S$ linear isomorphism, so we can without loss of generality assume that $\xi_1=(1, 0, \dots, 0)$.
Let $\varphi: K_S^n\to K_S^n/K_S\xi_1\cong K_S^{n-1}$
 be the natural quotient map. Since $I_S \xi_1$ is a cocompact lattice in
$K_S \xi_1$ and $\Gamma\subset K_S^n$ is discrete, the module $\varphi (\Gamma)$ is discrete and $\varphi(\xi_2),\dots, \varphi(\xi_m)$
are linearly independent over $I_S$. In view of the induction hypothesis, we have $\varphi(\xi_2),\dots, \varphi(\xi_m)$ are linear independent over $K_S$. Therefore $\xi_1,\ldots, \xi_m $ are linearly independent
over $K_S$.
\end{proof}

\begin{remark}\label{rem: hist} \new{The implication (2)$\Rightarrow$(3) holds without assuming that   $ \xi_1,\dots,  \xi_m$ belong to a discrete  $ I _S$-module, see \cite[Lemma 7.1]{KT}.} \end{remark}

For a discrete $I_S$-module  $\Gamma\subset K_S^n$
let  $K\Gamma$ (resp.~$K_S\Gamma$) be
 the    \new{$K$-linear (resp.~$K_S$-linear)}  span of $\Gamma$ in $K_S^n$.
We call the dimension over $K$ of $
K\Gamma $ the {\sl rank} of $\Gamma$.
It follows from Lemma \ref{lem: discrete module}  that the rank  of $\Gamma$  is less than or equal to $n$ and the equality holds
 if and only if
$\Gamma$  has finite covolume.

\smallskip

In the following lemma we prove a Gram-Schmidt orthogonalization  process for ultrametric
local  fields.
\begin{lemma}\label{lem: orthonormal basis}
Let $K_v$ be a ultrametric local field. For any $K_v$-linearly independent vectors $\xi_1,\dots ,\xi_m\in K_v^n$  there exist
linearly independent vectors $\eta_1,\dots, \eta_m \in K_v^n$ such that  $\eta_1,\dots, \eta_r$ are
in the $K_v$-linear span of $\xi_1,\dots, \xi_r$  for all $r\le m$, and
\begin{equation}\label{eq: orthonormal basis}
\|a_1 \eta_1+\dots+ a_m\eta_m\|_v=\max_{1\le i\le m}|a_i|_v\quad \mbox{for all}\  \  a_i\in K_v.
\end{equation}
\end{lemma}

\begin{remark}\label{rem;measure}
In the sequel we call a  basis of  $L_v\new{\df}\mathbf{span}_{K_v}\{\xi_1, \ldots, \xi_m \}$ which satisfies (\ref{eq: orthonormal basis})
an {\sl orthonormal basis} of $L_v$.  The map
\[
\varphi: K_v^m\to L_v
\quad \mbox{where} \quad
\varphi_v(a_1, \dots, a_m)=a_1 \eta_1+\dots+a_m \eta_m
\]
is an isometric embedding 
\new{sending} $\mathbf{vol}_v^m$ to $\mathbf{vol}_{L_v}$.
\end{remark}

\begin{proof}
\new{Contrary to the} archimedean case, here we
 choose an entry  with maximal absolute value for the corresponding  vector.
Write
\[
\xi_i=(x_{i1}, \dots, x_{in}), \quad \mbox{where} \quad x_{ij}\in K_v \quad{ and }\quad  1\le i\le m.
\]
First we choose $j_1\le n$ such that
$\|\xi_1\|_v=|x_{1j_1}|_v$ and set $\eta_1=x_{1j_1}^{-1}\xi_1$. Next we take
$\eta'_2 = \xi_2-x_{2j_1}\eta_1=(y_{21},\dots, y_{2n})$. We  choose $j_2\le n$ such that
$\|\eta'_2\|_v=|y_{2 j_2}|_v$ and set $\eta_2=y^{-1}_{2j_2}\eta'_2$. In general after $r$ steps we have $r$ different integers $j_1
\dots, j_{r}$ and unit norm vectors $\eta_1,\dots, \eta_{r}$ such that
$\eta_i$ has $j_{i}$-th entry $1$ and $j_s$-th entry zero for $s<i$.  We take
\[
 \eta'_{r+1}=\xi_{r+1}-x_{r+1,j_1} \eta_1-\dots x_{r+1,j_{r}}\eta_{r}=(y_1, \dots, y_n)
\]
and choose $j_{r+1}$ such that $\|\eta'_{r+1}\|_v=|y_{j_{r+1}}|_v$.
We set $\eta_{r+1}=y_{j_{r+1}}^{-1}\eta'_{r+1}$. Then it has $j_{r+1}$-th entry
$1$ and $j_s$-th entry $0$ for $s<r+1$.
This  induction process gives $m$ unit norm vectors $\eta_1,\dots, \eta_m$.

For $(a_1, \ldots, a_m)\in K_v^m$ let
\[
k=\min \{1\le r\le m: |a_{j_r}|_v=\max_{1\le i\le m} |a_i|_v  \}.
\]
It is clear from the construction that
\[
\|a_1 \eta_1+\dots+ a_m\eta_m\|_v=|a_{j_k}|_v,
\]
which proves (\ref{eq: orthonormal basis}).
\end{proof}


\new{The next lemma appeared as \cite[Corollary 8.4]{KT} for $K = \Q$ and as \cite[Corollary 5.8]{KTMP} for $K$  a number field.}

\begin{lemma}\label{lem: sum covolume}
Suppose  $\Gamma$ and $\Gamma'$ are discrete $I_S$-modules in  $K_S^n$ with \eq{empty}{K \Gamma \cap K \Gamma'=\{0\}.}
Then \begin{equation}\label{eq: inequality volume}
\mathbf{cov_r}(\Gamma+\Gamma')\le \mathbf{cov_r}(\Gamma)\mathbf{cov_r}(\Gamma').
\end{equation}
\end{lemma}

\begin{proof}
Let  $L=K_S\Gamma $, $L'=K_S\Gamma'$ and $L''=L+ L'$.
\new{In view of \equ{empty}}, Lemma \ref{lem: orthonormal basis} implies that
$L''$ is  a direct sum of $L$ and $L'$.
The right (resp.~left) hand side of (\ref{eq: inequality volume})  is the covolume of
$\Gamma+\Gamma'$ with relative  to $\mathbf{vol}_{L}\times  \mathbf{vol}_{L'}$ (resp. $\mathbf{vol}_{L''}$).
Let $L=\prod_{v\in S} L_v$ and $L'=\prod_{v\in S} L_v'$ according to (\ref{eq;decomposition}).
In view of (\ref{eq;measure}),
it suffices to prove that for each $v\in S$ there is a
positive Haar measure set $R_v$ of $L_v+L_v'$ such that
\begin{align}\label{eq;goal}
\mathbf{vol}_{L_v+L_v'} (R_v)\le (\mathbf{vol}_{L_v}\times  \mathbf{vol}_{L_v'})(R_v).
\end{align}

Let $r$ and $m$ be the rank of $L$ and $L''$ respectively.
For each $v\in S$ we choose an orthonormal  basis  $\xi_{1,v},\dots, \xi_{r,v}$ of $L_v$ and
an orthonormal basis $\xi_{r+1, v},
\dots, \xi_{m, v}$
of  $L'_v$.
 We will show that (\ref{eq;goal}) holds for
\begin{align*}
R_v:=\{  a_1 \xi_{1,v}+\ldots+a_{m}\xi_{m,v}: a_i\in B_1( K_v)  \}.
\end{align*}
For all $v\in S$
\begin{align}\label{eq;book}
(\mathbf{vol}_{L_v}\times  \mathbf{vol}_{L_v'})(R_v)= \mathbf{vol}_{v}^1 \big( B_1(K_v)\big)^{m}.
\end{align}
If $v$ is archimedean, then it is clear from Eulidean geometry (i.e.~volume of parallelotope) that
\[
\mathbf{vol}_{L_v+L_v'}(R_v)\le \mathbf{vol}_{v}^1 \big( B_1(K_v)\big)^{m}=(\mathbf{vol}_{L_v}\times  \mathbf{vol}_{L_v'})(R_v).
\]
If $v$ is ultrametric, we let
 $\eta_{1,v}, \ldots, \eta_{m,v}$ be an orthonormal basis of $L_v+L_v'$.
Then Remark \ref{rem;measure} implies that $R_v$ is contained in
\[
R_v':=\{ (a_1, \ldots, a_{m})\in B_1(K_v^{m}): a_1 \eta_1+\ldots+a_{m}\eta_m  \}.
\]
Using  Remark \ref{rem;measure} again \new{together with} (\ref{eq;book}),  we have
\[
\mathbf{vol}_{L_v+L_v'}(R_v)\le \mathbf{vol}_{L_v+L_v'}(R_v')=\mathbf{vol}_{v}^1 \big( B_1(K_v)\big)^{m}= (\mathbf{vol}_{L_v}\times  \mathbf{vol}_{L_v'})(R_v).
\]
\end{proof}

  \section{Successive minima}\label{sec: minima}
  The aim of this section is to prove Theorem \ref{thm: Minkowski}.

 \begin{lemma}\label{lem: estimate lattice}
 Let $\Gamma \subset K_S^n$ be a  discrete $I_S$-module with finite covolume
  and let $R\subset K_S^n$ be a measurable subset. Then there
  exists $\xi \in K_S^n$ such that
  \begin{align}
   \mathbf{card}\big( (\xi+R)\cap \Gamma\big)\ge  \mathbf {vol}(R)/\mathbf{cov}(\Gamma).
  \end{align}\label{eq;midterm}
 \end{lemma}

 \begin{proof}
 Let $\chi_R $ be the characteristic function of $R$, and let  $F\subset K_S^n$ be a fundamental domain \new{for}  $\Gamma$. Then
 \[
 \int_F \mathbf{card} \big((\xi+R)\cap \Gamma\big) \, \mathrm{d}\xi= \int_F \sum_{\gamma\in \Gamma}\chi_R(\gamma-\xi)d \xi
 =\mathbf{vol}(R).
 \]
 Therefore there exists $\xi\in F$ such that (\ref{eq;midterm}) holds.
 \end{proof}

 \begin{lemma}\label{lem: general Minkowski}
 Let $\Gamma\subset K_S^n$ be a discrete $I_S$-module with finite covolume.
 Let  $R_1$ be a  centrally symmetric convex  subset of
 $K_{P_0}^n$    and let  $R_2$ be a closed additive subgroup of $K_{S\smallsetminus P_0}^n$.
 Suppose $R\subset K_S^n$  is  \new{equal to} $
 R_1\times R_2
 $
 \new{with} the natural identification of $K_S^n$ with $K_{P_0}^n\times K_{S\smallsetminus P_0}^n$.
 If $\mathbf{vol}(R)>2^{n(\sigma+2\tau)} \mathbf{cov}(\Gamma)$, then $R$ contains
 a nonzero
 element of  $\Gamma$.
 \end{lemma}

 \begin{proof}
 Let  $R'=(\frac{1}{2}R_1)\times R_2$.  It follows
 from the assumption on the $\mathbf{vol}(R)$ that $\mathbf {vol}(R')> \mathbf{cov}(\Gamma)$.
 According to Lemma \ref{lem: estimate lattice}, we can find two distinct elements
  $\gamma_1, \gamma_2\in \Gamma$ and $\xi\in K_S^n$ such that $\gamma_i-\xi\in R' $ for $i=1, 2$.
  Therefore  the nonzero element $\gamma_1-\gamma_2$ \new{belongs to}
  $R$.
 \end{proof}

Recall that   $\lambda_m(\Gamma)$ $(1\le m\le n)$
 is  the $m$-th minimum of a discrete $I_S$-module  $\Gamma$, see (\ref{eq: intro minimum}).
It follows directly  from the  definition  that there \new{exist} $K$-linearly independent vectors $\xi_1,\dots,\xi_n\in \Gamma$ with
$$\|\xi_m\|=\lambda_m\text{ for all }1\le m\le n.$$
Moreover, by
Lemma \ref{lem: discrete module} these vectors are also linearly independent over $K_S$.

According to Lemma \ref{lem: general Minkowski} for any  $0<t  < \lambda_1(\Gamma)$ we have
\begin{equation}\label{eq;temp}
\Big(\prod_{v\in S\smallsetminus P_0} q_v^{-n}\Big)t^{n\sharp S}
\mathbf{vol}\big(B_1(K_S^n)\big)
{\le} \mathbf{vol}\big(B_t({K_S^n})\big)  \le 2^{n(\sigma+2\tau)}
\mathbf{cov}(\Gamma).
\end{equation}
Since  $B_1( \mathbb R^m)$ contains $ \{(x_1, \ldots, x_m)\in \mathbb R^m: -m^{-1/2}\le x_i\le m^{-1/2} \}$,
for archimedean $v\in P$
  we have
\begin{align}
\label{eq;simple}
\mathbf{vol}_v\big(B_1(K_v^n)\big)\ge 2^{n\varepsilon_v}  n^{-n\varepsilon_v/2},
\end{align}
where for complex place we use $B_1(\mathbb C^n)=B_1(\mathbb R^{2n})$.
By (\ref{eq;temp}), (\ref{eq;simple}) and (\ref{eq;ball}) we have
\begin{align}
\label{eq: Minkowski}
\lambda_1(\Gamma)^{n\sharp S}
  \le  n^{n(\sigma+2\tau)/2}\Big(\prod_{v\in S\smallsetminus P_0} q_v^{n}\Big)
\mathbf{cov}(\Gamma).
\end{align}

\begin{lemma}\label{lem: estimate}
Let $\Gamma$ be a discrete $I_S$-module with finite covolume.
Suppose that $\xi_1,\dots,\xi_n\in \Gamma$ are $K$-linearly independent vectors
and
$\|\xi_m\|=\lambda_m(\Gamma) $
for all  $1\le m\le n$.
Then there exists $g\in \GL_n(K_S)$ such that
\begin{equation}\label{eq: det cov}
\mathbf{cont}\big(\mathbf{det}(g)\big)=\prod_{i=1}^n\mathbf{cont}({\xi_i})^{-1},
\end{equation}
 and any nonzero vector of
$\Gamma'\new{\df}\Gamma g$ has norm greater than or equal to one.
\end{lemma}
\begin{proof}
Suppose that  $\xi_i=(\xi_{i,v})_{v\in S}$ where $\xi_{i,v}\in K_v^n$
(the notation here  is the same as   \S \ref{sec: preliminaries}).
By Lemma \ref{lem: discrete module},  for every $v\in S$ the vectors $\xi_{1, v}, \ldots, \xi_{n,v}$
are $K_v$-linearly independent in $K_v^n$.
Using  Gram-Schmidt orthogonalization process (see Lemma \ref{lem: orthonormal basis} for the ultrametric case), for each $v\in S$ we can find
 an orthonormal basis $\eta_{1,v},\dots,\eta_{n,v}$
 such that for every $1\le m\le n$ the $K_v$-linear span of $\eta_{1,v},\dots, \eta_{m,v}$
 is the same as
 that of $\xi_{1,v},\dots, \xi_{m,v}$.
  Let $b_i=(b_{i,v})_{v\in S}\in K_S$ ($1\le i\le n $)
  such that $|b_{i,v}|_v=\|\xi_{i,v}\|_v$.
  It follows from the definition of content  that
  \begin{align}\label{eq;content}
  \mathbf{cont}(b_i)= \mathbf{cont}(\xi_i).
  \end{align}
 Since
$\eta_i\new{\df}(\eta_{i,v})_{v\in S} \ (1\le i\le n)$ is a $K_S$-basis of $K_S^n$,
there is a unique  $g\in \GL_n(K_S)$ such that $ \eta_ig=b_i^{-1}\eta_i$.
We claim that this $g$ satisfies the requirement of the lemma.

The equation (\ref{eq: det cov}) follows from
\[
\mathbf{cont}\big(\mathbf{det}(g)\big)=\mathbf{cont}( b_1^{-1}\dots b_n^{-1})=\prod_{i=1}^n\mathbf{cont}(b_i)^{-1}=
\prod_{i=1}^n\mathbf{cont}(\xi_i)^{-1},
\]
where in the last equality we use (\ref{eq;content}).
For the other  conclusion suppose  that $
\zeta=c_1 \eta_1+\dots+c_m\eta_m\in \Gamma'=\Gamma g$
where $c_i\in K_S$ and $c_m\neq 0$.
We have
\[
\zeta g^{-1}=c_1 b_1 \eta_1+ \dots+c_mb_m\eta_m\in \Gamma.
\]
Since for every $v\in S$ the basis $\eta_{1,v},\dots, \eta_{n,v}$ is orthonormal, we have
\begin{align}\label{eq;zeta}
\|\zeta g^{-1}\|\le\|\zeta\| \max_{1\le i\le m}|b_i|=\|\zeta\|\cdot\lambda_m(\Gamma).
\end{align}
On the other hand for any $1\le j\le m$, the $K_S$-linear span of  $\eta_1, \dots, \eta_j$ is the same as that of $\xi_1,\dots, \xi_j$.
Since $c_mb_m\neq 0$, Lemma  \ref{lem: discrete module} implies that
$\xi_1, \ldots, \xi_{m-1}, \zeta g^{-1}$ are $K$-linearly independent.
\new{Thus} it follows from the definition of the $m$-th minimum of $\Gamma=\Gamma' g^{-1}$
that $\|\zeta g^{-1}\|\ge \lambda_m(\Gamma)$.
This estimate together with (\ref{eq;zeta}) imply
 $\|\zeta\| \ge 1$, which  completes the proof.
 \end{proof}

\begin{theorem}\label{thm: refine}
Let $\Gamma$ be a discrete $I_S$-module with finite covolume.  Let $\xi_1,\dots, \xi_n\in\Gamma$
be $K$ linearly independent vectors with
 $\|\xi_m\|=\lambda_m(\Gamma)$
 for all $1\le m\le n$. Then we have
 \begin{align}\label{eq: precise}
\mathbf{cov}(I_S^n)^{-1}{\mathbf{cov} (\Gamma)}\le  \prod_{i=1}^n\mathbf {cont}(\xi_i)\le n^{n(\sigma+2\tau)/2}\Big( \prod_{v\in S\backslash P_0}q_v^{n}\Big){\mathbf{cov} (\Gamma)}.
 \end{align}
\end{theorem}
\begin{proof}
We first  prove the upper bound  of (\ref{eq: precise}).
\new{Suppose that}  $\Gamma'=\Gamma g$, where  $g\in \GL_n(K_S)$, satisfies the conclusion of Lemma \ref{lem: estimate}.
Then
$\lambda_1(\Gamma')\ge 1$.
Applying  (\ref{eq: Minkowski}) for   $\Gamma'$ we have
\begin{align}\label{eq;tmd}
1\le \lambda(\Gamma')^{n\sharp S}\le  n^{n(\sigma+2\tau)/2}\Big(\prod_{v\in S\smallsetminus P_0} q_v^{n}\Big)
\mathbf{cov}(\Gamma').
\end{align}
On the other hand by (\ref{eq;det}) and (\ref{eq: det cov})
\begin{align}\label{eq;tmd1}
{\mathbf{cov} (\Gamma')}=\mathbf{cov} (\Gamma)\cdot\mathbf{cont}\big(\mathbf{det}(g)\big)=
\mathbf{cov} (\Gamma)\cdot\prod_{i=1}^n \mathbf{cont}(\xi_i)^{-1}.
\end{align}
The upper bound  of (\ref{eq: precise}) follows from   (\ref{eq;tmd}) and (\ref{eq;tmd1}).

 Let $\Gamma''$ be the $I_S$-linear span of $\xi_1, \ldots, \xi_n$. Since $\Gamma''$ is a submodule of $\Gamma$, by
  Lemma \ref{lem: sum covolume}  we get
\[
\mathbf{cov}(\Gamma)\le \mathbf{cov}(\Gamma'')\le \prod_{i=1}^n\mathbf{cov_r}(I_S\xi_i)=
\mathbf{cov}(I_S^n)\cdot\prod_{i=1}^n
\mathbf{cont}(\xi_i),
\]
which implies  lower bound of (\ref{eq: precise}).
\end{proof}

To prove Theorem \ref{thm: Minkowski} we need a balance  between contents and
 norms of vectors in  $K_S^n$. The following lemma is a generalization of \cite[Lemma \new{8.6}]{KT} \new{and \cite[Lemma 5.9]{KTMP},}
 and the proof is the same.

\begin{lemma}\label{lem: balance}
For any $\xi\in K_S^n$ with $\mathbf {cont}(\xi)\neq 0$, there exists $a\in I_S^*$ such that $\|a \xi\|^{\sharp S}\asymp
\mathbf {cont} (\xi)$ where the implied constants depend on $K$ and $S$.
\end{lemma}
\begin{proof}
Suppose that $S=\{v_1, \dots, v_m\}$ where $m=\mathbf{card} (S)$.
Let $\mathbb R_+$ be the multiplicative group of positive real numbers.
 We define a   map
\[
\varphi: K\to \mathbb R^m_+ \quad \mbox{by}\quad
\varphi(a)=( |a|_{v_1},\dots, |a|_{v_m}).
\]
Let
\[
H=\{(r_1,\dots, r_m)\in \mathbb R^m_+: \prod_{i=1}^m  r_i^{\varepsilon_{v_i}}=1 \}.
\]
It follows from Dirichlet's unit theorem (see \cite[Chapter II \S 18]{CF}) that the group
$\varphi(I_S^*)\subset H$
is a cocompact lattice in $H$.  So there exists $A\ge 1$
which  depends on $K$ and $S$ such that for any  $(r_1,\dots, r_m)\in H$
we can find $a\in I_S^*$ with
\begin{equation}\label{eq: balance 1}
A^{-1}\le  r_i |a|_{v_i}\le A.
\end{equation}

Suppose  $\xi=(\xi_{v})_{v\in S}$. It follows from \new{the} definition that
 \[
 ( {\|\xi_{v_1}\|_{v_1}} \cdot  \mathbf {cont}(\xi)^{-1/(m+\tau)}, \ldots, {\|\xi_{v_m}\|_{v_m}} \cdot  \mathbf {cont}(\xi)^{-1/(m+\tau)} )\in H.
 \]
By (\ref{eq: balance 1}) \new{one} can find $a\in I_S^*$ such that for all $1\le i\le m$
\[
A^{-1}\le {\|\xi_{v_i}\|_{v_i}}{\mathbf {cont}(\xi)^{-1/(m+\tau)}}|a|_{v_i}\le A.
\]
Therefore
\[
A^{-m-\tau}\mathbf {cont}(\xi)\le \|a\xi \|^{m+\tau}\le A^{m+\tau} \mathbf {cont}(\xi).
\]
\end{proof}

\begin{proof}[Proof of Theorem \ref{thm: Minkowski}]
Let $\xi_1,\dots, \xi_n\in \Gamma$
be $K$-linearly independent vectors with $\|\xi_i\|=\lambda_i(\Gamma)$.
By  Theorem \ref{thm: refine}
\begin{align}\label{eq;easy}
\prod_{i=1}^n\mathbf{cont}(\xi_i)\asymp \mathbf{cov}(\Gamma),
\end{align}
where the implied constants depend on $K, S$ and $n$.
The definitions of  content and norm  imply
\begin{align}\label{eq;easy1}
\mathbf{cont}(\xi_i)\le \lambda_i(\Gamma)^{\sharp S}.
\end{align}
According to  Lemma \ref{lem: balance} there exists $a_1, \ldots, a_n\in I_S^*$
such that
 \begin{align}\label{eq;easy2}
\mathbf{cont}(\xi_i)\gg \|a_i\xi_i\|^{\sharp S},
 \end{align}
 where the implied constant depends on $K$ and $S$.
Note that \new{elements}  $a_1\xi_1, \ldots, a_n\xi_n\in \Gamma$ are linear independent over $K$.
So the  definition
of successive  minima implies
\begin{align}\label{eq;easy3}
\prod_{i=1}^n \lambda_i(\Gamma)\le \prod_{i=1}^n\|a_i\xi_i\|.
\end{align}
Therefore the conclusion of Theorem \ref{thm: Minkowski} follows from (\ref{eq;easy}),
(\ref{eq;easy1}), (\ref{eq;easy2}) and (\ref{eq;easy3}).
\end{proof}

\section{Mahler's compactness criterion}

Let $X=\SL_n(I_S)\backslash \SL_n(K_S)$. There is a one-to-one correspondence between
$X$ and
\[
\{I_S^n g: g\in \SL_n(K_S) \}
\]
via the map $  \SL_n(I_S)g\to I_S^n g   $.
In this section $\mathbf{e}_1, \ldots, \mathbf {e}_n$ denotes the standard basis of $K_S^n$,
i.e.~$\mathbf{e}_i$ has $i$-th entry $1$ and other entries $0$.
Before proving Theorem  \ref{thm: intro main}
we need  the following lemma. \new{See \cite[Corollary \new{8.6}]{KT} for $K = \Q$ and  \cite[Corollary 5.1]{KTMP} for $K$ a number field.}

\begin{lemma}\label{lem: finite}
Let  $M>0$.
Then there are only finitely many  $I_S$-submodules $\Gamma$ of $I_S^n$
such that $\mathbf{card}(I_S^n/\Gamma)\le M$.
\end{lemma}

\begin{proof}
Let $\Gamma\subset I_S^n$ be an $I_S$-submodule with $\mathbf{card}(I_S^n/\Gamma)\le M$.
 For every $1\le i\le n$ there is an ideal $J_i$ of $I_S $ such that
\[
I_S \mathbf e_i \cap \Gamma= J_i \mathbf e_i\quad  \mbox{and }\quad \mathbf{card} (I_S/J_i)\le M.
\]
Therefore
\[
J_1\times \dots \times J_n\subset  \Gamma \subset I_S^n.
\]
Note that $I_S$ is a Dedekind domain. It follows from the structure theory of ideals in $I_S$ that
 there are only finitely many ideals $J$ in $I_S$ such that $\mathbf{card} (I_S/J)\le M$. So the
conclusion of the lemma holds.
\end{proof}

\begin{proof}[Proof of Theorem \ref{thm: intro main}]
Let $\pi: \SL_n(K_S)\to X$ be  the natural quotient  map and let
\begin{equation}\label{eq: Mahler set}
r=\inf \{\|\xi g\|: \xi \in I_S^n,\  \xi\neq 0, \ g\in \SL_n(K_S),\  \pi(g)\in R  \}.
\end{equation}

Suppose  $R$ is relatively compact. There exists a relatively compact subset  $F\subset \SL_n(K_S)$  with
$\pi(F)=R$. 
\new{Therefore} there exists  $C>0$ such that
\begin{align}\label{eq;compact}
\|\xi g\|\le C\|\xi\| \quad \mbox{for every}\ \  \xi\in K_S^n \ \  \mbox{and}\ \  g\in F.
\end{align}
The discreteness of $\Gamma$ and (\ref{eq;compact}) imply $r>0$.

Now we assume $r>0$  and   prove that  $R$ is relatively  compact.
Let $\{g_i\}_{i \ge  1}$   be a  sequence in $\pi^{-1}(R)$.
It suffices to show that there exists $g\in \SL_n(K_S)$   such that
$\pi(g) $ is a limit point of a  subsequence of  $\{\pi(g_i) \}_{i\ge 1}$.
By Theorem \ref{thm: Minkowski}
 there exists  $C \ge 1$ such that  for   any free $I_S$-module  $\Gamma\in R$  one has
\begin{equation}\label{eq: Mahler 0}
r \le \lambda _1(\Gamma)\le \lambda_n(\Gamma)\le C.
\end{equation}

For every  $i\ge 1$ let $\xi_1^{(i)}, \dots,
\xi_n^{(i)} \in I_S^n$
be $K$-linearly independent  vectors  such that    $\|\xi_j^{(i)}g_i\|$ equals to the $j$-th minimum of
$I_S^n g_i$.
By (\ref{eq: Mahler 0}) we have
\begin{align}\label{eq;beijing}
\|\xi_j^{(i)}g_i\|\le C\quad \forall\  i\ge 1\mbox{ and }1\le j\le n.
\end{align}
Let
$$\Gamma_i=\mathbf{span}_{I_S}\{\xi_j^{(i)}: 1\le j\le n\}.$$

According to  (\ref{eq: Mahler 0}) and Theorem \ref{thm: Minkowski} there exists $M>0$ such that
 \begin{equation}\label{eq: Mahler 1}
\mathbf{cov}(I_S^n)\le \mathbf{cov}(\Gamma_i)= \mathbf{cov}(\Gamma_i g_i)\le M\quad \forall \  i\ge 1.
\end{equation}
By Lemma \ref{lem: finite} and (\ref{eq: Mahler 1}), the set $\{\Gamma_i: i \ge  1\}$ is finite.
Therefore by possibly passing to a subsequence we may assume that  there exists
$h\in \GL_n(K)$
  such that
   $\Gamma_{i}=I_S^nh$ for  all $i\ge 1$.
It follows that there  is a sequence $\{f_i\}_{i\ge 1}$ in $\GL_n(I_S)$ such that  $\mathbf{e}_j f_i h=\xi_j^{(i)}$ for all $i\ge 1$
and $1\le j\le n$.
By (\ref{eq;beijing}) there is a subsequence $\{ g_{i_k} \}_{k\ge 1}$ of $\{g_i \}_{i\ge 1}$
and  $g\in \GL_n(K_S) $ such that
\begin{align}\label{eq;beijing1}
f_{i_k}h g_{i_k}\to g \quad \mbox{as}\quad k\to \infty.
\end{align}
Since $\mathbf{det}$ is continuous and  $I_S$ is discrete in $K_S$, for $k$ sufficiently large we have
$\mathbf{det}(f_{i_k})=\mathbf{det}(f_{i_{k+1}})\in I_S^*$. Therefore by possibly passing to a subsequence
and multiplying the first  row of  $h$  by some element of $I_S^*$, we assume without loss of generality that
$f_{i_k}\in \SL_n(I_S)$ for all $k$.

The group  $ h^{-1}\SL_n(I_S)h\cap \SL_n(I_S)$ has finite index in
$h^{-1}\SL_n(I_S)h$.  So by possibly passing to a subsequence
we can find $ f\in h^{-1}\SL_n(I_S)h$ and a sequence $\{h_k \}_{k\ge 1}$ of $\SL_n(I_S)$ such that
\begin{align}\label{eq;beijing2}
h^{-1}f_{i_k} h = fh _k\quad \forall\   k\ge 1.
\end{align}
By (\ref{eq;beijing1}) and (\ref{eq;beijing2}) we have
$h f h_k g_{i_k}\to g$ as $k\to \infty$.
Therefore $$h_k g_{i_k}\to f^{-1}h^{-1} g\text{ as }k\to \infty.$$
Since $h_k\in \SL_d(I_S)$ we have
$\pi(g_{i_k})\to \pi (f^{-1}h^{-1} g)$ as $k\to \infty$.
\end{proof}

\end{document}